\documentclass[11pt,twoside]{amsart}
\usepackage[T1]{fontenc}
\usepackage[utf8]{inputenc}
\usepackage[english]{babel}
\usepackage{graphicx}
\usepackage{amsmath}
\usepackage{amssymb}
\usepackage{enumerate}
\usepackage[all,cmtip]{xy}

\usepackage{amsthm}
\usepackage{amsfonts}
\usepackage{changepage}

\newtheorem{thm1}{Theorem}[section]
\newtheorem{theorem}[thm1]{Theorem}
\newtheorem{lemma}[thm1]{Lemma}
\newtheorem{corollary}[thm1]{Corollary}
\newtheorem{proposition}[thm1]{Proposition}

\theoremstyle{definition}
\newtheorem{definition}[thm1]{Definition}

\RequirePackage{ifthen}
\RequirePackage{calc}
\newcounter{exampleendflag}
\newcommand{\exendhere}{
  \setcounter{exampleendflag}{0} 
  \ifmmode
    \eqno
    \ensuremath{\blacktriangle}
  \else
    \hspace{\stretch{1}}
    \ensuremath{\blacktriangle}
  \fi
}
\newenvironment{example}{
  \setcounter{exampleendflag}{1}
  \begin{exx}
}{
  \ifthenelse{\value{exampleendflag}=1}{\exendhere}{} 
  \end{exx}
}

\newtheorem{exx}[thm1]{Example}

\theoremstyle{remark}
\newtheorem{remark}[thm1]{Remark}

\title{Total blow-ups of modules and universal flatifications}
\author{\gss}
\thanks{The author is supported by the Swedish Research Council, grant number 2011-5599.}

\newcommand{\saeden}{S\ae d\'en}
\newcommand{\stahl}{St\aa hl}
\newcommand{\gss}{Gustav \saeden\ \stahl}
\newcommand{\fF}{\mathcal{F}}

\newcommand{\fR}{\mathcal{R}}
\newcommand{\sym}{\operatorname{Sym}}
\newcommand{\spec}{\operatorname{Spec}}
\newcommand{\proj}{\operatorname{Proj}}

\newcommand{\im}{{\operatorname{im}}} 
\renewcommand{\hom}{\operatorname{Hom}}

\renewcommand{\O}{\mathcal{O}}

\renewcommand{\phi}{\varphi}
\renewcommand{\epsilon}{\varepsilon}

\renewcommand\O{\includegraphics[height=8.5pt]{EGAO}}
\renewcommand\O{\mathcal{O}}
\newcommand{\E}{\mathcal{E}}
\renewcommand{\L}{\mathcal{L}}

\newcommand{\Bl}{\mathbb{B}}

\newcommand{\C}{\mathbb{C}}

\renewcommand{\P}{\mathbb{P}}

\newcommand{\p}{\mathfrak{p}}

\newcommand{\grass}{\mathcal{G}\!{\kern 0.0667em}\textit{rass}}

\makeatletter
\newcommand{\extp}{\@ifnextchar^\@extp{\@extp^{\,}}}
\def\@extp^#1{\mathop{\bigwedge\nolimits^{\!#1}}}
\makeatother

\numberwithin{equation}{section}
\address{Department of Mathematics, KTH Royal Institute of Technology, SE-100 44 Stockholm, Sweden}
\email{gss@math.kth.se}
\begin{document}
\begin{abstract}
We study the projective spectrum of the Rees algebra of a module, and characterize it by a universal property. As applications, we give descriptions of universal flatifications of modules and of birational projective morphisms. 
\end{abstract}
\maketitle

\section*{Introduction}
The blow-up of a scheme along a closed subscheme is a fundamental object in algebraic geometry. Constructing the blow-up is done by taking the projective spectrum of the Rees algebra of the corresponding ideal sheaf. 
There are nowadays many algebraic results concerning the generalized concept of \emph{Rees algebras of modules}, see for instance \cite{reesbud} and \cite{reesmodules}. We will here describe the geometry behind this object, and characterize it by a universal property. 
There is already a notion of the blow-up of a module, see \cite{MR2188879}, defined as the blow-up of an ideal isomorphic to a torsionfree quotient of the determinant of the module. This blow-up is used to flatify the module and is by construction birational. This is not the geometric object we will study here, but we will relate our results to this notion in Section~\ref{sec:flat}.

Instead, we will study another transformation associated to modules, 
namely the projective spectrum of the Rees algebra of a module. We call  this scheme the \emph{total blow-up} of the module, and we denote the total blow-up of a module $M$ by $\Bl(M)$. As these concepts generalize naturally to coherent sheaves on locally noetherian schemes, we will study the total blow-up in this setting. Let therefore $\fF$ be a coherent sheaf on a locally noetherian scheme $X$, and suppose that there is some open $U\subseteq X$ such that $\fF|_U$ is locally free. We will then, in Theorem~\ref{thm:main}, characterize the total blow-up as having the following property: For every morphism $f\colon Y\to X$ such that $f^{-1}(U)$ is schematically dense, there is an equality of sets 
\[\{\text{quotients }f^\ast\fF\twoheadrightarrow \L \text{ with } \L \text{ invertible}\}=\hom_X\bigl(Y,\Bl(\fF)\bigr).\]
This result tells us that the total blow-up $\Bl(\fF)$ has the same property as the projective space bundle $\P(\fF)=\proj\bigl(\sym(\fF)\bigr)$, but on a more restricted class of morphisms.  We also show that the total blow-up gives a connection between the projective space bundle and the classical blow-up of an ideal. Indeed, if $\E$ is a locally free sheaf on $X$, then $\Bl(\E)=\P(\E)$ is the projective space bundle, and if $\mathcal{I}$ is an ideal sheaf, then $\Bl(\mathcal{I})=\operatorname{Bl}_\mathcal{I}X$ is the classical blow-up along $\mathcal{I}$. 

In Section~\ref{sec:flat}, we apply our results to flatifications of modules, related to the theory of Raynaud and Gruson \cite{MR0308104}. There are many other results on this subject, and we show that these results are naturally explained using the total blow-up. Related is the work of Oneto and Zatini \cite{MR1218672} who showed that the Nash transformation of a module gives a flatification, and that this transformation is equal to a blow-up of a certain fractional ideal. This was later rediscovered by Villamayor \cite{MR2188879} who gave another description of this blow-up and called it the blow-up of a module. See also \cite{MR3042637}. We show in Corollary~\ref{cor:nashh} that this blow-up is equal to the total blow-up of the determinant of the module. As another application, we show in Theorem~\ref{thm:birproj} that any birational projective morphism can be described as the total blow-up of the push-forward of a very ample line bundle.

Finally, in Section~\ref{sec:tor}, we introduce the concept of torsionless quotients of algebras to motivate and explain our previous assumptions and results. We show that the Rees algebra can be characterized as the torsionless quotient of the symmetric algebra, and that our assumptions are needed for the torsionless quotient to behave well under base change. 

\section{Rees algebras of modules and torsionless quotients}
In this paper, $A$ will always denote a noetherian ring, and $M$ a finitely generated module over $A$. In this setup, we will study the geometry of the Rees algebra of a module defined by Eisenbud, Huneke and Ulrich in \cite{reesbud}. 
Their definition of the Rees algebra of an $A$-module $M$ is as the quotient
\[\fR(M)=\sym(M)/\cap_gL_g\]
where $L_g=\ker\bigl(\sym(g)\colon\sym(M)\to\sym(E)\bigr)$ and $g$ runs over all maps from $M$ to all free modules $E$. 
When $M$ has a generic rank, that is, when $M\otimes_AQ(A)$ is a free module over the total ring of quotients $Q(A)$ of $A$, then this is equivalent to the definition given in \cite{reesmodules}.  

The easiest method of calculating the Rees algebra is via the use of versal maps. A map $f\colon M\to F$ is called \emph{versal} if $F$ is free and any map $M\to E$ where $E$ is free factors via~$f$. This is equivalent, see \cite[Proposition~1.3]{reesbud}, to the dual $F^\ast\to M^\ast$ being surjective. We refer also to \cite{jag2} for another characterization of versal maps. When $A$ is noetherian and $M$ is finitely generated, there is always a versal map $M\to F$ that can be constructed by choosing a surjection from a free module $F'$ onto the dual $M^\ast$. Then, letting $F=(F')^\ast$, we get a versal map from the composition $M\to M^{\ast\ast}\hookrightarrow (F')^\ast=F$.
The Rees algebra of a module $M$ can then easily be calculated as the image of the induced map $\sym(M)\to\sym(F)$.  Another description of Rees algebras of modules is due to the following result which we will use to generalize the notion of Rees algebras to sheaves in Section~\ref{sec:3}.
\begin{theorem}[{\cite[Theorem~4.2]{jag1}}]\label{thm:div}
Let $A$ be a noetherian ring and let $M$ be a finitely generated $A$-module. The Rees algebra $\fR(M)$ of $M$ 
is equal to the image of the canonical graded $A$-algebra map
\[\sym(M)\to\Gamma(M^\ast)^\vee\]
where 
\[\Gamma(M^\ast)^\vee=\bigoplus_{n\ge0}\hom_A\Bigl(\Gamma^n\bigl(\hom_A(M,A)\bigr),A\Bigr)\] denotes the graded dual of the algebra of divided powers $\Gamma(M^\ast)$ of the dual of the module $M$.
\end{theorem}

Rees algebras of modules are strongly related to the notion of torsionless modules, see \cite{reesbud} and \cite{jag1}. A module is called torsionless if it embeds into a free module. The torsionless quotient $M^{tl}$ of a module $M$ is the image of the canonical map $M\to M^{\ast\ast}$. It is easy to see that the torsionless quotient embeds into a free module and is equal to $M$ if $M$ itself is torsionless. Given a versal map $M\to F$, it factors as $M\to M^{\ast\ast}\hookrightarrow F$ so the torsionless quotient of $M$ is isomorphic to the image of $M\to F$. For Rees algebras, we have that $\fR(M)=\fR(M^{tl})$. In the cases we will study here, there is another useful description of the torsionless quotient. 
\begin{proposition}\label{prop:cartier}
Let $M$ be a finitely generated module over a noetherian ring $A$, and suppose that there is an injective flat map $A\to A'$ such that $M'=M\otimes_AA'$ is locally free over $A'$. Then, the torsionless quotient of $M$ is isomorphic to the image of the induced map ${M\to M'}$. 
\end{proposition}
\begin{proof}
Take a versal map $M\to F$. Since $A\to A'$ is flat, it follows that $M'\to F'=F\otimes_AA'$ is versal. Moreover, $M'$ is locally free over $A'$, so the identity map $M'\to M'$ factors as $M'\to F'\to M'$, implying that $M'\to F'$ is injective. Hence, we get a commutative diagram
\[\xymatrix{
M\ar[r]\ar[d]& F\ar@{^(->}[d]\\
M'\ar@{^(->}[r] &F'}\]
and, since the image of $M\to F$ is isomorphic to the torsionless quotient, it follows that $M^{tl}=\im(M\to M')$. 
\end{proof}

\section{Associated points}
We will here provide some classical results on the relation between schematically dense open subsets and associated points, as well as study the associated primes of the Rees algebra.
\begin{proposition}[{\cite[Proposition~9.19 and Proposition~9.22]{wedhorn}}]\label{prop:assdense}
Let $X$ be a locally noetherian scheme, and let $j\colon U\hookrightarrow X$ be an open subscheme of $X$. 
Then $U$ is schematically dense if and only if $U$ contains all associated points of $X$, which is also equivalent to the homomorphism $\O_X\to j_\ast j^\ast \O_X$ being injective.
\end{proposition}

Let now $X=\proj\bigl(\fR(M)\bigr)$. The associated points of $X$ are the associated primes of the sheaf~$\O_X$. These correspond to the associated primes of $\fR(M)$ that do not contain the irrelevant ideal $\fR_+(M)=\bigoplus_{n\ge1}\fR^n(M)$.
The associated primes of the Rees algebra have a complete characterization given by the following result.
\begin{proposition}[{\cite[Proposition~1.5]{reesbud}}]
The associated primes of $\fR(M)$ are in a one-to-one correspondence with the associated primes of $A$.
\end{proposition}
The proof of this fact uses, given a versal map $\phi\colon M\to F$, that any associated prime of $\fR(M)$ is the contraction of an associated prime of the polynomial ring $\sym(F)$ which in turn is an extension of an associated prime of $A$. Thus, an associated prime ideal $\p\subset A$ gives an associated prime $\p\sym(F)$ of $\sym(F)$ which contracts to an associated prime ideal $\p'=\p\sym(F)\cap\fR(M)$ of $\fR(M)$. As both the Rees algebra and the symmetric algebra are generated in degree~1, it follows that $\p'$ contains the irrelevant ideal $\fR_+(M)$ precisely when $\phi^{-1}(\p F)=M$. Using this description of the associated points of the Rees algebra, we can give the following result on the projective spectrum of the Rees algebra of $M$, which is useful for the results of Section~\ref{sec:3}.


\begin{proposition}\label{prop:assofrees}
Let $M$ be a finitely generated module over a noetherian ring $A$, and let $X=\proj\bigl(\fR(M)\bigr)$ with structure morphism $\pi\colon X\to \spec(A)$. Given any open subset ${U\subseteq\spec(A)}$, the inverse image $\pi^{-1}(U)$ is schematically dense in $X$ if and only if ${\hom_A(M,A)=\hom_A(M,\mathfrak{p})}$ for every associated prime ideal $\mathfrak{p}\in\spec(A)\setminus U$. In particular, if $U$ is schematically dense, then $\pi^{-1}(U)$ is schematically dense.
\end{proposition}
\begin{proof}
Take a versal map $\phi\colon M\to F$. From Proposition~\ref{prop:assdense}, we have that $\pi^{-1}(U)$ is schematically dense if and only if it contains all associated points of $X$. From the description above, the associated points of $X$ are precisely the associated primes of $\fR(M)$ that do not contain the irrelevant ideal. Furthermore, the associated prime ideals of $\fR(M)$ that contain the irrelevant ideal are precisely the ones that correspond to associated prime ideals $\p$ in $A$ that fulfill $\phi(M)\subseteq \p F$. 

If $\phi(M)\subseteq\p{F}$, then, since any map $M\to A$ factors as $M\to \p F\hookrightarrow F\to A$, it follows that any map $M\to A$ takes values in $\p$. For the converse, if $\hom_A(M,A)=\hom_A(M,\mathfrak{p})$, then, for any projection $F\twoheadrightarrow A$, we have that $M\to F\twoheadrightarrow A$ takes values in $\p$ and the result follows. 

Finally, if $U$ is schematically dense, then there are no associated prime ideals in $\spec(A)\setminus U$ so the condition is trivially satisfied.
\end{proof}

\section{Total blow-ups of modules}\label{sec:3}
We will here consider the projective spectrum of the Rees algebra of a module $M$, which we will call the \emph{total blow-up of $M$} and denote by $\Bl(M)=\proj\bigl(\fR(M)\bigr)$. It turns out that these concepts naturally generalizes to sheaves. Indeed, Theorem~\ref{thm:div} tells us that the Rees algebra can be described in terms of the symmetric algebra and duals of the algebra of divided powers. Both the symmetric algebra \cite[Proposition~III.6.7]{bourAlg} and the algebra of divided powers \cite[Th\'eor\`eme~III.3]{Roby1963} commute with arbitrary base change. Furthermore, duals commute with flat base change \cite[\S0.5.7.6]{ega1ny}, so it follows that the description of the Rees algebra given in Theorem~\ref{thm:div} can be glued to a sheaf, see also \cite[Remark~4.4]{jag1}.
\begin{definition}
Let $\fF$ be a coherent sheaf on a locally noetherian scheme $X$. We define the \emph{Rees algebra $\fR(\fF)$ of the coherent sheaf $\fF$} as the image of the canonical map
\[\sym(\fF)\to\Gamma(\fF^\ast)^\vee\]
of $\O_X$-algebras.
Moreover, we define the \emph{total blow-up of $\fF$} as the projective spectrum of $\fR(\fF)$ and denote it by $\Bl(\fF)=\proj\bigl(\fR(\fF)\bigr)$.
\end{definition}
\begin{remark}
It is easy to see that versal maps commute with flat base change, so we could have defined the Rees algebra of a coherent sheaf as locally being in terms of the description given in \cite{reesbud}. However, the definition we chose here has the advantage of being written in terms of a global description. Note that the only two facts that we will use about the Rees algebra in the sequel is that  the symmetric algebra surjects onto it together with Proposition~\ref{prop:assofrees}.
\end{remark}
A classical result is that $\P(\fF)=\proj\bigl(\sym(\fF)\bigr)$ equals the Grassmannian $\grass_1(\fF)$ parametrizing locally free quotients of $\fF$ of rank~$1$. Given any $\fF$  there is a natural surjection $\sym(\fF)\twoheadrightarrow\fR(\fF)$ giving a closed embedding ${\Bl(\fF)\hookrightarrow\P(\fF)}$.

\begin{theorem}\label{thm:main}
Let $\fF$ be a coherent sheaf on a locally noetherian scheme $X$, with total blow-up $\pi\colon\Bl(\fF)\to X$. Suppose that there is some $U\subseteq X$ such that $\fF|_U$ is locally free. Let $f\colon Y\to X$ be a morphism such that $f^{-1}(U)$ is schematically dense in $Y$. Then, to give a morphism $Y\to\Bl(\fF)$ over $X$ is equivalent to give a locally free quotient $f^\ast\fF\twoheadrightarrow\mathcal{L}$ of rank~$1$.
\end{theorem}
\begin{proof}
Let $f\colon Y\to X$ be a morphism of schemes such that $f^{-1}(U)$ is schematically dense. 

For one direction, take a morphism $Y\to\Bl(\fF)$ over $X$. Composing this with the closed embedding $\Bl(\fF)\hookrightarrow\P(\fF)$ gives a locally free rank $1$ quotient of $f^\ast\fF$. 

Conversely, consider a quotient $f^\ast\fF\twoheadrightarrow\mathcal{L}$ where $\mathcal{L}$ is a line bundle. This is equivalent to a morphism $g\colon Y\to \P(\fF)$ over $X$. The restriction of $g$ to $f^{-1}(U)$ gives a map ${f^{-1}(U)\to p^{-1}(U)}$ where $p\colon\P(\fF)\to X$ denotes the structure morphism. As $\fF$ is locally free on $U$, we have that $p^{-1}(U)=\pi^{-1}(U)$.
\[\xymatrix@C=0.2em{
&\pi^{-1}(U)\ar@{}[rr]|-*[@]{\subseteq}\ar@{=}[d]&&\Bl(\fF)\ar@/^2pc/[dd]^-\pi\ar@{^(->}[d]\\
&p^{-1}(U)\ar@{}[rr]|-*[@]{\subseteq}&&\P(\fF)\ar[d]^-p\\
f^{-1}(U)\ar@{}[r]|-*[@]{\subseteq}\ar[ur]&Y\ar[rr]^-f\ar[urr]^-g&&X}\]
 Thus, $g\bigl(f^{-1}(U)\bigr)\subseteq\pi^{-1}(U)\subseteq\Bl(\fF)$, so $f^{-1}(U)\subseteq g^{-1}(\Bl(\fF)$. Since $\Bl(\fF)\hookrightarrow\P(\fF)$ is a closed embedding and $f^{-1}(U)$ is schematically dense 
it follows that $g$ factors via $\Bl(\fF)$.
\end{proof}

\begin{example}
If $\fF$  is a locally free sheaf of rank $n$, then we can choose $U=X$ and get ${\Bl(\fF)=\P(\fF)}$, the projective space bundle of $\fF$.
\end{example}

\begin{example}
When $\fF=\mathcal{I}$ is an ideal sheaf and $U=X\setminus V(\mathcal{I})$, this gives the classical universal property of blow-ups in closed subschemes. Indeed, if $f\colon Y\to X$ is a morphism with $f^{-1}(U)$ schematically dense, we have by Proposition~\ref{prop:assdense} a commutative diagram
\[\xymatrix{
f^{\ast}\mathcal{I}\ar[r]\ar[d]&\O_Y\ar@{^(->}[d]\\
j_\ast j^\ast f^\ast\mathcal{I}\ar@{=}[r]&j_\ast j^\ast\O_Y}\]
where $j\colon f^{-1}(U)\hookrightarrow Y$ denotes the inclusion. Thus, the result of Proposition~\ref{prop:cartier} tells us that ${\mathcal{I}\O_Y=\im(f^\ast\mathcal{I}\to\O_Y)}$ is isomorphic to the torsionless quotient of $f^\ast\mathcal{I}$. By Theorem~\ref{thm:main}, a map $Y\to\Bl(\mathcal{I})$ gives a surjection $f^\ast\mathcal{I}\twoheadrightarrow\L$ for some line bundle $\L$. As any such map factorizes via the torsionless quotient, this induces a surjection $\mathcal{I}\O_Y\twoheadrightarrow\L$. Hence, $\mathcal{I}\O_Y$ is an ideal sheaf that surjects onto a line bundle, so $\mathcal{I}\O_Y$ is invertible. Thus, $\mathcal{I}\O_Y$ is a Cartier divisor.
\end{example}

The reason that we restrict ourselves to maps $f\colon Y\to X$ with $f^{-1}(U)$ schematically dense in the assumptions of Theorem~\ref{thm:main} is firstly because a similar assumption is needed for the universal property of the classical blow-up. Indeed, in the classical case we only consider maps $Y\to X$ that turns our ideal sheaf to a Cartier divisor. The complement of a Cartier divisor is schematically dense, so our assumption is in fact weaker since the converse is not true, that is, the complement of a schematically dense subset is not necessarily a Cartier divisor. The second reason for these restrictions is related to the notion of torsionless quotients of algebras studied in Section~\ref{sec:tor}.


Note that Theorem~\ref{thm:main} states that $\Bl(\fF)$ has the same property as $\P(\fF)$, but only for morphisms $f\colon Y\to X$ such that $f^{-1}(U)$ is schematically dense. 
Therefore, this does \emph{not} imply that $\Bl(\fF)=\P(\fF)$ since neither $\pi\colon\Bl(\fF)\to X$ nor $p\colon\P(\fF)\to X$ need to have the property that the inverse image of $U$ is schematically dense. 
However, that $\pi^{-1}(U)$ is schematically dense in $\Bl(\fF)$ is a weaker condition than $p^{-1}(U)$ being schematically dense in $\P(\fF)$, so $\pi$ is in some sense closer than $p$ to being universal in this class of morphisms. Proposition~\ref{prop:assofrees} gives a description of when $\pi^{-1}(U)$ is schematically dense and therefore describes when the property of $\Bl(\fF)$ is universal. For instance, if $U$ is schematically dense, then $\pi^{-1}(U)$ is also, which is the case considered in \cite{reesmodules}. 

In general, we let $\Bl_{U}(\fF)$ denote the closure of $\pi^{-1}(U)$ in $\Bl(\fF)$. 
\begin{corollary}\label{cor:main2}
For the class of morphisms $f\colon Y\to X$, such that $f^{-1}(U)$ is schematically dense, the scheme $\pi_U\colon\Bl_U(\fF)\to X$ has the universal property that to give a morphism $Y\to \Bl_U(\fF)$ over $X$ is equivalent to give a a locally free quotient $f^\ast\fF\twoheadrightarrow\mathcal{L}$ of rank $1$. 
\end{corollary}
\begin{proof}
That an $X$-morphism $Y\to\Bl_U(\fF)$ is equivalent to a quotient $f^\ast\fF\twoheadrightarrow\mathcal{L}$ follows by the same argument as in the proof of Theorem~\ref{thm:main}. Moreover, by construction, $\pi_U\colon\Bl_{U}(\fF)\to X$ has the property that $\pi_U^{-1}(U)$ is schematically dense, so $\pi_U$ is a member of this class of morphisms. 
 \end{proof}

\section{Applications to flattening of modules and birational morphisms}\label{sec:flat}
In this section, we will apply our results to get a simple description of flattening of modules, using the results of Raynaud and Gruson. We will start by briefly recalling some of their results and for more details we refer to \cite[\S5.2]{MR0308104}. Given a projective morphism $f\colon X\to S$ of noetherian schemes and a coherent sheaf $\fF$ on $X$, we can consider the Quot scheme $Q=\mathcal{Q}uot(\fF)$ parametrizing flat quotients of $\fF$ \cite[\S3]{MR1611822}. 
Suppose that $\fF$ is flat over an open $U\subseteq S$, and define $S'$ as the scheme theoretic closed image of the induced morphism $U\to Q$ \cite[\S6.10]{ega1ny}. This gives a projective morphism $g\colon S'\to S$  such that $g^{-1}(U)$ is schematically dense and the quotient $g^\ast\fF\twoheadrightarrow \mathcal{Q}$, corresponding to the morphism $S'\to Q$, is flat. Moreover, the morphism $g$ is universal with this property. 

From now on, we specialize ourselves to the case $f\colon X\to X=S$ being the identity. Recall that a locally free sheaf of finite rank is flat and assume that $\fF$ is locally free of finite rank $d$ on $U\subseteq X$. In this case, the Quot scheme specializes to the Grassmannian $G_d=\grass_d(\fF)$ parametrizing locally free quotients of $\fF$ of rank~$d$.
This is also the setup of \cite{MR1218672} where the \emph{Nash transformation} $S^N(\fF)$ of $\fF$ is defined as the closure of the image of $U$ in $G_d$ with the reduced structure. We will instead use the term Nash transformation for the scheme theoretic closed image explained above and also use the notation $S^N(\fF)=S'$.
Formulating the result of \cite{MR0308104} in the case of Grassmannians, we then get the following.
\begin{proposition}\label{prop:nash}
For the class of morphisms $f\colon Y\to X$, such that $f^{-1}(U)$ is schematically dense, the Nash transformation has the universal property that a morphism $Y\to S^N(\fF)$ is equivalent to a locally free quotient $f^\ast\fF\twoheadrightarrow\mathcal{Q}$ of rank $d$. 
\end{proposition}
\begin{remark}
This result is strongly related to the theory of torsionless quotients. Indeed, let $j\colon U\to X$ denote the inclusion. Writing Proposition~\ref{prop:cartier} in terms of sheaves gives that the torsionless quotient of $\fF$ equals the image of the morphism $\fF\to j_\ast j^\ast\fF$. The restriction $j^\ast \fF$ of $\fF$ to 
$U$ is locally free of rank $d$, so the quotient~$\mathcal{Q}$ of Proposition~\ref{prop:nash} is in fact unique and equal to the torsionless quotient of $\fF$. Thus, the quotient $\mathcal{Q}$ is the flatification of $\fF$.
\end{remark}
The Plücker embedding gives a closed immersion of $G_d$ in $\P\bigl(\extp^d\fF\bigr)=G_1$ \cite[\S9.8]{ega1ny}.
As $\fF$ is locally free of rank $d$ on $U$, we also have that the determinant $\extp^d\fF$ is locally free of rank~$1$ on $U$, giving a morphism $U\to G_1$, which by construction equals the composition $U\to G_d\hookrightarrow G_1$. Thus, we get a Nash transformation $S^N\bigl(\extp^d\fF\bigr)$ of $\extp^d\fF$ as a closed subscheme of $\P\bigl(\extp^d\fF\bigr)$. Since the Plücker embedding is a closed immersion, it follows that $S^N(\fF)=S^N\bigl(\extp^d\fF\bigr)$, see also \cite[Property~1.4]{MR1218672}. 

\begin{corollary}\label{cor:nashh}
The Nash transformation $S^N(\fF)$ of a coherent sheaf $\fF$ is isomorphic to the scheme $\Bl_U\bigl(\extp^d\fF\bigr)$. 
\end{corollary}
\begin{proof}
From above we have that $S^N(\fF)=S^N\bigl(\extp^d\fF\bigr)$. By comparing Corollary~\ref{cor:main2} and Proposition~\ref{prop:nash} we see that $\Bl_U\bigl(\extp^d\fF\bigr)$ and $S^N\bigl(\extp^d\fF\bigr)$ have the same universal property, so the result follows.
\end{proof}
\begin{remark}
Proposition~\ref{prop:assofrees} gives a description for when $\Bl(\fF)=\Bl_U(\fF)$. In particular, they are equal if $U$ is schematically dense in $X$. Thus, the result above tells us that if $\fF$ is locally free of rank $d$ on a schematically dense open $U\subseteq X$, then the flatification of $\fF$ is given by the total blow-up of the determinant of $\fF$.
\end{remark}
There are results that shows that the flatification of a module is equal to a blow-up of a certain fractional ideal, see \cite{MR1218672}, \cite{MR2188879} and \cite{MR3042637}. More precisely, given a module~$M$ that is locally free of rank $d$ on some schematically dense and open $U\subseteq\spec(A)$, they consider the fractional ideal  given as the image of a map $\extp^d M\to Q(A)$, where $Q(A)$ denotes the total quotient ring of $A$. The Rees algebra of a module sees only the torsionless quotient, that is, ${\fR(M)=\fR(M^{tl})}$ for any module $M$, so taking the total blow-up of $\extp^d M$ is equivalent to taking the total blow-up of $\bigl(\extp^dM\bigr)^{tl}$. As the torsionless quotient of $\extp^dM$ is isomorphic to the fractional ideal described above, we recover their description as an immediate consequence.

Another application of total blow-ups is to describe birational projective morphisms, as we explain below.  For integral schemes, Grothendieck showed that such a morphism is equal to the blow-up of a fractional ideal \cite[Th{\'e}or{\`e}me~(2.3.5)]{MR0217085}. The following result is a generalization of this, in terms of total blow-ups, where we drop the integral assumption. 
\begin{theorem}\label{thm:birproj}
Let $f\colon Y\to X$ be a projective morphism of locally noetherian schemes. Assume that there is a schematically dense open subset $U\subseteq X$ such that $f^{-1}(U)$ is schematically dense in $Y$ and that the restriction $f^{-1}(U)\xrightarrow{\raisebox{-0.25 em}{\smash{\ensuremath{\sim}}}} U$ is an isomorphism. Then, $Y$ is isomorphic to the total blow-up of $f_\ast\mathcal{L}$, for any very ample line bundle $\mathcal{L}$ on $Y$. 
\end{theorem}
\begin{proof}
As $Y$ is projective over $X$, there is by definition a very ample line bundle $\mathcal{L}$ on $Y$. 
Very ample line bundles are globally generated, so $f^\ast f_\ast\mathcal{L}\to\mathcal{L}$ is surjective, and the resulting morphism $Y\hookrightarrow\P(f_\ast\mathcal{L})$ over $X$ is a closed embedding \cite[Proposition~(4.4.4)]{MR0217084}.


Let $\pi\colon\Bl(f_\ast\mathcal{L})\to X$ denote the total blow-up of $f_\ast\mathcal{L}$. The assumption that there is an isomorphism $f^{-1}(U)\xrightarrow{\raisebox{-0.25 em}{\smash{\ensuremath{\sim}}}} U$ implies that the sheaf ${f_\ast\mathcal{L}|_U=\mathcal{L}|_{f^{-1}(U)}}$ is locally free of rank~$1$. As $U$ is schematically dense in $X$,  Proposition~\ref{prop:assofrees} tells us that $\pi^{-1}(U)$ is schematically dense in  $\Bl(f_\ast\mathcal{L})$.  These two observations imply that $\Bl(f_\ast\mathcal{L})=\Bl_U(f_\ast\mathcal{L})$ has the universal property of Corollary~\ref{cor:main2}. 
 Thus, the surjection $f^\ast f_\ast\mathcal{L}\to\mathcal{L}$ gives a morphism $Y\to\Bl(f_\ast\mathcal{L})$ over $X$.  Furthermore, the restriction of $\pi\colon\Bl(f_\ast\mathcal{L})\to X$ to $\pi^{-1}(U)$ gives a morphism ${\pi^{-1}(U)\to U\cong f^{-1}(U)}\subseteq Y$ over $X$. 
Since $Y\hookrightarrow\P(f_\ast\mathcal{L})$ is a closed embedding and $\pi^{-1}(U)$ is schematically dense, it follows that ${\Bl(f_\ast\mathcal{L})\hookrightarrow\P(f_\ast\mathcal{L})}$ factors via a morphism $\Bl(f_\ast\mathcal{L})\to Y$. 

Hence, we have constructed $X$-morphisms in both directions between $Y$ and $\Bl(f_\ast\mathcal{L})$, and the universal property of $\Bl(f_\ast\mathcal{L})$ implies that they are inverses of each other.
\end{proof}

\section{Torsionless quotients of algebras}\label{sec:tor}
We will end by relating our results to a concept of torsionless quotients of algebras. This will also motivate why we restrict ourselves to morphisms $f\colon Y\to X$ such that $f^{-1}(U)$ is schematically dense for some open $U\subseteq X$ in the assumptions of Theorem~\ref{thm:main}. 

An equivalent definition of the torsionless quotient of a module, similar to the definition of the Rees algebra from \cite{reesbud}, is as the quotient $M^{tl}=M/\cap_g\ker(g\colon M\to E)$, where $g$ runs over all maps $g\colon M\to E$ for all free $A$-modules $E$.
Analogously, we now make the following definition.
\begin{definition}
Let $B$ be an $A$-algebra. We define the \emph{torsionless quotient of $B$} as the $A$-algebra
\[B^{tl}=B/\cap_g\ker(g\colon B\to P)\]
where $g$ runs over all $A$-algebra homomorphisms $B\to P$ for all flat $A$-algebras $P$.
\end{definition}

\begin{proposition}
Let $M$ be a finitely generated $A$-module. Then $\sym(M)^{tl}=\fR(M)$.
\end{proposition}
\begin{proof}
A classical result is that any flat module $P$ over $A$ is a direct limit of free modules of finite rank, and any homomorphism $M\to P$ factors as $M\to E\to P$, for some free module $E$ of finite rank \cite[Th\'eor\`eme~1.2]{MR0254100}. Choosing a versal map $\phi\colon M\to F$, it follows that any homomorphism $M\to P$ factors as ${M\to F\to E\to P}$. 

Thus, any $A$-algebra homomorphism $\sym(M)\to P$, for any flat $A$-algebra $P$, factors as $\sym(M)\to\sym(F)\to P$. Hence, $\ker\bigl(\sym(\phi)\bigr)\subseteq\ker\bigl(\sym(M)\to P\bigr)$. As $\sym(F)$ is flat, it follows that $\sym(M)^{tl}=\sym(M)/\ker\bigl(\sym(\phi)\bigr)=\fR(M)$. 
\end{proof}
\begin{remark}
This result gives a new definition of the Rees algebra of a module $M$ as the torsionless quotient of the symmetric algebra of $M$.
\end{remark}

\begin{proposition}\label{prop:olablad}
Let $B$ be an $A$-algebra. Suppose that there is an injective flat map $A\to A'$
such that $A'\to B\otimes_AA'$ is flat. Then $B^{tl}=\im(B\to B\otimes_AA')$.
\end{proposition}
\begin{proof}
Consider a map $B\to P$ where $P$ is flat. Tensoring $A\hookrightarrow A'$ with $P$ preserves the injectivity, so $P\hookrightarrow P\otimes_AA'$ is injective. We get the commutative diagram
\[\xymatrix{
B\ar[r]\ar[d]&P\ar@{^(->}[d]	\\
B\otimes_AA'\ar[r]&P\otimes_AA'}\]
so $\ker(B\to B\otimes_AA')\subseteq\ker(B\to P)$. Since $B\otimes_AA'$ is flat over $A'$, the result follows.
\end{proof}

\begin{remark}
Perhaps a more natural definition of the torsionless quotient would be in terms of smooth algebras rather than flat algebras, as these are the usual algebra analogues of (locally) free modules. For our purposes, however, flat algebras will be sufficiently restrictive. 
\end{remark}

We will now prove some results on torsionless quotients of base changes of the Rees algebra, which will help explain the results of the previous sections.

\begin{lemma}\label{lem:3434}
Let $M$ be a finitely generated $A$-module, and suppose that there is an injective flat map $A\hookrightarrow A'$ such that $M\otimes_AA'$ is locally free. Then, $\fR(M)\to\fR(M\otimes_AA')$ is injective.
\end{lemma}
\begin{proof}
Choosing a versal map $M\to F$ gives an injection $\fR(M)\hookrightarrow\sym(F)$. Furthermore, as shown in \cite[Proposition~1.3]{reesbud}, versal maps are preserved under flat base change  so ${M\otimes_AA'\to F\otimes_AA'}$ is versal and $\fR(M\otimes_AA')\hookrightarrow\sym(F\otimes_AA')$ is injective. Moreover, $\sym(F)$ is flat, so $\sym(F)\hookrightarrow\sym(F\otimes_AA')$ is injective. Thus, we get a commutative diagram
\[\xymatrix{
\fR(M)\ar@{^(->}[r]\ar[d]&\sym(F)\ar@{^(->}[d]\\
\fR(M\otimes_AA')\ar@{^(->}[r]&\sym(F\otimes_AA')}\]
which implies that $\fR(M)\to\fR(M\otimes_AA')$ is injective.
\end{proof}

\begin{proposition}\label{prop:fix}
Let $M$ be a finitely generated $A$-module, and suppose that there is a flat map $A\to A'$ such that $M\otimes_AA'$ is locally free over $A'$. If $A\to B$ is a ring homomorphism such that $B\to B'=A'\otimes_AB$ is injective and flat, then $\fR(M\otimes_AB)$ is the torsionless quotient of $\fR(M)\otimes_AB$. In particular, there is a surjective map
$\fR(M)\otimes_AB\twoheadrightarrow\fR(M\otimes_AB)$.
\end{proposition}
\begin{proof}
Since $M\otimes_AA'$ is locally free, we have that $\fR(M\otimes_AA')=\sym(M\otimes_AA')$, and we get, by Proposition~\ref{prop:olablad}, that the torsionless quotient of $\fR(M)\otimes_AB$ is given by the image of the map
\[\fR(M)\otimes_AB\to\fR(M\otimes_AA')\otimes_AB=\sym(M\otimes_AA')\otimes_AB= \sym(M\otimes_AA'\otimes_AB).\]
Furthermore, the map $\fR(M\otimes_AB)\to\fR(M\otimes_AB\otimes_BB')=\sym(M\otimes_AB\otimes_BB')$ is injective by Lemma~\ref{lem:3434}. 
Since $M\otimes_AA'\otimes_AB=M\otimes_AB\otimes_BB'$, we have a commutative diagram
\[\xymatrix{
\sym(M)\otimes_AB\ar@{->>}[r]\ar@{=}[d]&\fR(M)\otimes_AB\ar[r]&\sym(M\otimes_AA'\otimes_AB)\ar@{=}[d]\\
\sym(M\otimes_AB)\ar@{->>}[r]&\fR(M\otimes_AB)\ar@{^(->}[r]&\sym(M\otimes_AB\otimes_BB')}\]
which induces a surjective map $\fR(M)\otimes_AB\twoheadrightarrow\fR(M\otimes_AB)$, and the result follows.
\end{proof}
\begin{remark}
We chose to write the previous proposition in the algebraic setting to reflect the rest of this section, but it is perhaps more illuminating to consider it geometrically. From this 
point of view the result says that if $M$ is locally free on an open ${U=\spec(A')\subseteq\spec(A)}$ and ${f\colon\spec(B)\to\spec(A)}$ is a morphism such that $f^{-1}(U)=\spec(B')$ is open and schematically dense, then there is a surjective map $f^{\ast}\fR(M)\twoheadrightarrow\fR(f^\ast M)$.
\end{remark}
In general, there is not even a canonical map from $\fR(M)\otimes_AB\to\fR(M\otimes_AB)$, as the following example will show. 
\begin{example}\label{ex:nomap}
Consider the ring $A=\C[x]/(x^2)$, the $A$-module $M=(x)$, and the $A$-algebra $B=A[S]/(xS)$. Then we have
\[\fR(M)\otimes_AB=A[T]/(xT,T^2)\otimes_AA[S]/(xS)=A[S,T]/(xS,xT,T^2).\]
However, $M\otimes_AB\cong B/(x)=\C[S]$ and 
\[\fR(M\otimes_AB)=B[T]/(xT)=A[S,T]/(xS,xT),\]
so there is no canonical map $\fR(M)\otimes_AB\to\fR(M\otimes_AB)$.
\end{example} 
We will now relate the theory of this section to the total blow-up. Consider the case ${Y=\spec(B)}$ and $X=\spec(A)$ with a finitely generated module $M$ over $A$. The assumptions of Theorem~\ref{thm:main} requires that there is an open $U\subseteq\spec(A)$ such that $M|_U$ is locally free and that $f^{-1}(U)$ is schematically dense. We can now see that this requirement ensures that the assumptions of Proposition~\ref{prop:fix} are satisfied. Hence, in this case, $\fR(M\otimes_AB)$ is the torsionless quotient of $\fR(M)\otimes_AB$. This is an implicit reason for why $Y\to\P(M)$ factorizes via $\Bl(M)$. Indeed, an $X$-morphism $Y\to\P(M)$ gives a line bundle quotient $f^\ast M\to L$, which in turn gives a $Y$-morphism $Y\to\Bl(f^\ast M)=\proj\bigl(\fR(f^\ast M)\bigr)$. By  
Proposition~\ref{prop:fix} we get a $Y$-morphism $Y\to\proj\bigl(\fR(f^\ast)\bigr)\hookrightarrow \proj\bigl(f^\ast\fR(M)\bigr)$.  Since $\proj\bigl(f^\ast\fR(M)\bigr)=\Bl(M)\times_XY$, this gives an $X$-morphism $Y\to\Bl(M)$ factorizing $Y\to\P(M)$. 

In Example~\ref{ex:nomap} we saw that there is not even a map $\fR(M)\otimes_AB\to\fR(M\otimes_AB)$ in general, which suggests the necessity of the assumptions of Theorem~\ref{thm:main}.

\section*{Acknowledgements}
I would like to thank David Rydh for all our invaluable discussions and Runar Ile for his careful reading and helpful comments.

\bibliography{references}{}
\bibliographystyle{amsalpha}

\end{document}